\documentclass[twoside,11pt]{article}
\usepackage[english]{babel}
\usepackage{amsmath}
\usepackage{amsthm}
\usepackage{amsfonts}
\usepackage{amssymb}
\usepackage{graphicx}
\usepackage{enumerate}
\usepackage{mathrsfs}
\usepackage{graphicx}

\makeatletter
\newcommand*\bigcdot{\mathpalette\bigcdot@{.5}}
\newcommand*\bigcdot@[2]{\mathbin{\vcenter{\hbox{\scalebox{#2}{$\m@th#1\bullet$}}}}}
\makeatother

 \theoremstyle{plain}
\newtheorem{theorem}{Theorem}[section]
\newtheorem{lemma}[theorem]{Lemma}
\newtheorem{proposition}[theorem]{Proposition}
\theoremstyle{remark}
\newtheorem{remark}[theorem]{\bf Remark}

\numberwithin{equation}{section}
\usepackage{color}
\usepackage[colorlinks,
            linkcolor=blue,
            anchorcolor=blue,
            citecolor=red,
            ]{hyperref}

\newcommand{\be}{\begin{equation}}
\newcommand{\ee}{\end{equation}}
\newcommand{\eps}{\varepsilon}

\allowdisplaybreaks[3]

\begin{document}

\title{
\bf{Incompressible and fast rotation limits for $3$D compressible rotating Euler system with general initial data
}}

\author{
Mikihiro Fujii$^\dag$, \qquad Yang Li$^\ddag$, \qquad Pengcheng Mu$^\S$ \\  \\
$^\dag$Graduate School of Science, \\ Nagoya City University, Nagoya, 467-8501, Japan \\ Email: fujii.mikihiro@nsc.nagoya-cu.ac.jp \\ \\
$^\ddag$School of Mathematical Sciences, \\ Anhui University, Hefei, 230601, People's Republic of China \\ Email: lynjum@163.com \\ \\
$^\S$School of Mathematical Sciences, \\ Anhui University, Hefei, 230601, People's Republic of China \\ Email: mupc024@163.com \\ \\
}

\maketitle

{\centerline {\bf Abstract }}
\vspace{2mm}
{
This paper is concerned with the low Mach and Rossby number limits of $3$D compressible rotating Euler equations with ill-prepared initial data in the whole space. More precisely, the initial data is the sum of a $3$D part and a $2$D part. With the help of a suitable intermediate system, we perform this singular limit rigorously with the target system being a $2$D QG-type. This particularly gives an affirmative answer to the question raised by Ngo and Scrobogna [\emph{Discrete Contin. Dyn. Syst.}, 38 (2018), pp. 749-789]. As a by-product, our proof gives a rigorous justification from the $2$D inviscid rotating shallow water equations to the $2$D QG equations in whole space.

}

\vspace{2mm}


\vspace{2mm}


\tableofcontents

\section{Introduction and main results}

\subsection{Background and motivation}
We study the singular limit of strong solution to the $3$D rotating isentropic compressible Euler equations:
\begin{equation}\label{el}
\begin{cases}
\partial_t\rho+\operatorname{div}(\rho u)=0,\\
\partial_{t}(\rho u)+\operatorname{div}(\rho u\otimes u)+\frac{1}{\varepsilon}\mathbf{e}_3\times \rho u+\frac{1}{\delta^2}\nabla P(\rho)=0.
\end{cases}
\end{equation}
Here, the unknowns $\rho$ and $u=(u_1,u_2,u_3)^{\top}$ are the density and velocity field of the fluid, respectively. The pressure $P(\rho)$ is given by $P(\rho)=\gamma^{-1}\rho^{\gamma}$ with $\gamma>1$, and $\mathbf{e}_3=(0,0,1)^{\top}$ is the rotating axis. The independent variables are $t\in\mathbb{R}^{+}$,
$x=(x_h,x_3)$ with $x_h=(x_1,x_2)\in\mathbb{R}^2$ and $x_3\in\mathbb{R}$. The system \eqref{el} is written in non-dimensional form where $\varepsilon$ is the Rossby number representing the ratio of the displacement due to inertia to the displacement duo to Coriolis force, and $\delta$ is the Mach number denoting the ratio of the characteristic speed of the fluid to the speed of sound. In this paper, we use $\mathbb{R}^2$ to denote the $2$D horizontal plane, and use $\nabla_h$, $\operatorname{div}_h$ to denote the gradient and divergence operators acting only in horizontal directions, respectively.

It is well-known that, in large scale motion of geophysical fluids the Rossby number $\eps$ and the Mach number $\delta$ are both very small. Thus, it is meaningful to rigorously justify the quasi-geostrophic (QG) limit of \eqref{el} in which the two parameters vanish simultaneously in the following way:
\begin{equation}\label{lim}
\eps\to0,~~\delta\to0,~~\mathrm{while}~~\frac{\delta}{\eps}=\mathrm{constant}.
\end{equation}
For convenience, we set
\begin{equation}\label{8202}
\bar{\gamma}:=\frac{\gamma-1}{2},~~\mathrm{and}~~\nu:=\frac{\delta}{\bar{\gamma}\varepsilon},
\end{equation}
which means in the setting of this paper $\nu$ is a fixed positive constant. Exactly as \cite{Ngo-DCDS}, we use the substitution
\begin{equation*}
\rho=\big(\bar{\gamma}(1+\delta b)\big)^{1/\bar{\gamma}},
\end{equation*}
and after a few straightforward calculations reformulate \eqref{el} as
\begin{equation}\label{CEL}
\partial_t U+\frac{\bar{\gamma}}{\delta}\mathcal{L}U+\mathcal{N}(U,\nabla U)=0,
\end{equation}
where
\begin{equation}\label{e9182}
U=\begin{pmatrix}
b\\u
\end{pmatrix},~~
\mathcal{L}U=
\begin{pmatrix}
\operatorname{div}u\\
\nu\mathbf{e}_3\times u+\nabla b
\end{pmatrix}
,~~\mathrm{and}~~
\mathcal{N}(U,\nabla U)
=
\begin{pmatrix}
u\cdot\nabla b+\bar{\gamma}b\operatorname{div}u\\
(u\cdot\nabla) u+\bar{\gamma}b\nabla b
\end{pmatrix}
.
\end{equation}
In this paper, we supplement \eqref{CEL} with the initial data
\begin{equation}\label{in}
U|_{t=0}=U_0(x)=\begin{pmatrix}
b_{0}\\u_{1,0}\\u_{2,0}\\u_{3,0}\\
\end{pmatrix}(x_{h},x_3)+\begin{pmatrix}
b_0^L\\u^L_{1,0}\\u^L_{2,0}\\u^L_{3,0}\\
\end{pmatrix}(x_{h})
\end{equation}
such that
\begin{equation}\label{eq161}
\begin{split}
&b_0\in H^{m}(\mathbb{R}^{3}),~~u_0:=(u_{1,0},u_{2,0},u_{3,0})^{\top}\in H^{m}(\mathbb{R}^{3}),\\
&b^L_0\in H^{m+3}(\mathbb{R}^{2}),~~u^L_{h,0}:=(u^L_{1,0},u^L_{2,0})^{\top}\in H^{m+3}(\mathbb{R}^{2}),~~\mathrm{and}~~u^L_{3,0}\in H^{m+3}(\mathbb{R}^{2}).\\
\end{split}
\end{equation}
To simplify the presentation,
we assume that $b_0$, $u_0$, $b^L_0$, $u^L_{h,0}$ and $u^L_{3,0}$ are independent of $\delta$.

We recall some previous results concerning the low Mach and Rossby number singular limits of the $3$D compressible rotating fluids. In \cite{B-F6}, E. Fereisl, I. Gallagher and A. Novotn\'{y}  first investigated the QG limit of global weak solutions to the compressible rotating Navier-Stokes equations with ill-prepared data and complete slip boundary conditions in an infinite slab $\mathbb{R}^2\times (0,1)$. Using the celebrated RAGE theorem, they proved the dispersion and justified the convergence towards the $2$D viscous quasi-geostrophic equation of the original system. After that, the incompressible and fast rotation limits of compressible rotating fluids in an infinite slab were studied in many works, such as \cite{F-2016,B-F4,B-F2,B-F1,B-F3,B-K1} for viscous fluids, and \cite{B-C1,B-T1} for inviscid fluids.

However, in the whole space case, the study of incompressible and fast rotation limits is slightly different. We first mention the work of M. Caggio and \v{S}. Ne\v{c}asov\'{a} \cite{M.Ca}. By combining the relative entropy method, the dispersion of acoustic wave, and the fact that fast rotation can weaken the nonlinear effects in Euler equation and thus extend the lifespan of the solution, the convergence of weak solutions of rotating isentropic compressible Navier-Stokes system to the strong solution of $3$D rotating incompressible Euler equation on arbitrary time interval was proved in \cite{M.Ca} as the Mach number and the viscosity coefficient tend to zero and the Rossby number is small enough but fixed. However, the limit as the Rossby number tends to zero was not explored in \cite{M.Ca}.
As far as we know, the first work concerning low Mach and Rossby number limits of the compressible rotating fluids in $3$D whole space was done by V-S. Ngo and S. Scrobogna in \cite{Ngo-DCDS}. In the aforementioned paper, the authors derived the Strichartz-type estimates for the linearized system of \eqref{CEL} in which the Rossby number and the Mach number are identical (i.e., the regime of QG limit) and the frequency is localized in $\mathcal{C}_{r,R}=\{\xi\in\mathbb{R}^3:|\xi|\leq R,|\xi_h|\geq r,|\xi_3|\geq r\}$, and identified zero as the QG limit of strong solutions to \eqref{CEL} with initial data of finite energy in $\mathbb{R}^3$, i.e., the $2$D part $(b^L_0,u^L_{1,0},u^L_{2,0},u^L_{3,0})$ in \eqref{in} equals to zero.

In fact, since the only element that has finite energy in $\mathbb{R}^3$ and belongs to the kernel of $\mathcal{L}$ is zero, the limit of solution to \eqref{CEL} with finite energy in $\mathbb{R}^3$ is, as it was shown in \cite{Ngo-DCDS}, also zero as $\delta\to0$. {\emph{However, if the initial data is the sum of a $3$D part and a $2$D part like \eqref{in}, then the solution to \eqref{CEL} will not tend to zero but converge to a $2$D system.}} More precisely, as $\eps,\delta$ tend to zero in the regime \eqref{lim}, the $3$D compressible rotating Euler equations \eqref{CEL} with the initial data \eqref{in} will converge to the equations:
\begin{equation}\label{Lim-sys}
\begin{cases}
\partial_{t}(\mathrm{curl}_h u^L_h-\nu b^L)+u^L_h\cdot\nabla_h (\mathrm{curl}_h u^L_h-\nu b^L)=0,\\
\nu u^{L}_h=\nabla_h^{\bot} b^L,\\
\partial_t u^L_3+u^L_h\cdot\nabla_h u^L_3=0,
\\
(b^L,u^L_h,u^L_3)|_{t=0}=(-\nu\omega^L_0,-\nabla_h^{\bot}\omega^L_0,u^L_{3,0}),
\end{cases}
\end{equation}
where $\omega^L_0=(\nu^2-\Delta_h)^{-1}(\mathrm{curl}_h u^L_{h,0}-\nu b^L_0)$. Observe that $(b^L,u_h^L)$ is the solution to the $2$D quasi-geostrophic equations (see \cite{B1}), and $u_3^L$ is governed by the $2$D transport equation. The aim of this paper is to rigorously justify this asymptotic limit for local strong solution to \eqref{CEL} with ill-prepared initial data of the type \eqref{in}.

Now we state the main difficulties and strategies of this paper. First, due to the reduction of dimension, directly showing the convergence of \eqref{CEL}-\eqref{in} to \eqref{Lim-sys} in Sobolev spaces on $\mathbb{R}^3$ seems challenging. To circumvent this difficulty, we first introduce the following equations as the intermediate system:
\begin{equation}\label{11183}
\begin{cases}
\partial_t a+w_h\cdot\nabla_h a+\bar{\gamma}a \operatorname{div}_h w_h+\frac{\bar{\gamma}}{\delta}\operatorname{div}_h w_h=0,\\
\partial_t w_h+(w_h\cdot\nabla_h) w_h+\bar{\gamma}a \nabla_h a+\frac{\bar{\gamma}}{\delta} \nu w_h^{\bot}+\frac{\bar{\gamma}}{\delta}\nabla_h a=0,\\
\partial_t w_3+w_h\cdot\nabla_h w_3=0,\\
\partial_{x_3}w_h=0,~~~\partial_{x_3}w_3=0,\\
(a,w_h,w_3)|_{t=0}=(b^L_0,u^L_{h,0},u^L_{3,0})^{\top}.
\end{cases}
\end{equation}
It is clear that \eqref{11183}$_{1,2}$ are the $2$D compressible rotating Euler equations, while \eqref{11183}$_3$  is the $2$D transport equation. The local well-posedness of strong solution to \eqref{11183} is standard and is stated in Lemma \ref{12084} below. Now, setting
\begin{equation*}
w:=(w_h,w_3)^{\top},~~\vartheta:=b-a,~~v:=u-w,
\end{equation*}
and subtracting \eqref{11183} from \eqref{CEL}, we get the following $3$D perturbed system which can be seen as a $3$D compressible rotating Euler equations supplemented with some linear terms if we regard $(a,w)$ as given functions for the moment:
\begin{equation}\label{11191}
\begin{cases}
\partial_t \vartheta+v\cdot\nabla \vartheta+\bar{\gamma}\vartheta \operatorname{div}v+\frac{\bar{\gamma}}{\delta}\operatorname{div}v=-w\cdot\nabla \vartheta-v\cdot\nabla a-\bar{\gamma}a\operatorname{div}v-\bar{\gamma}\vartheta\operatorname{div}w,\\
\partial_t v+(v\cdot\nabla) v+\bar{\gamma}\vartheta \nabla \vartheta+\frac{\bar{\gamma}}{\delta}\nu \mathbf{e}_3\times v+\frac{\bar{\gamma}}{\delta}\nabla \vartheta=-(w\cdot\nabla) v-(v\cdot\nabla) w-\bar{\gamma}a \nabla \vartheta-\bar{\gamma}\vartheta \nabla a,\\
(\vartheta,v)|_{t=0}=(b_0,u_0)^{\top}.
\end{cases}
\end{equation}
We first address the unique existence of the local strong solution to \eqref{11191} by smooth approximations and establish its uniform estimates with respect to $\eps$ and $\delta$. 
Next, we study the asymptotic limit from the intermediate system \eqref{11183} to the target system \eqref{Lim-sys}. This process is the QG limit of $2$D compressible rotating Euler equations. To this end, we split the solution into the fast oscillating wave and the slow part. By establishing suitable Strichartz estimates, the fast oscillating wave is shown to vanish as $\delta\rightarrow 0$. The convergence of slow part to the target system \eqref{Lim-sys} is proved by the energy estimate.

\subsection{Main results}
Before stating the main theorems of this paper, we give a lemma concerning the well-posedness and uniform estimates of local strong solutions to the intermediate system \eqref{11183}, which can be obtained by using the theory established by S. Klainerman and A. Majda \cite{K-M-1,K-M-2} for symmetric hyperbolic system and the classical results for transport equation:

\begin{lemma}\label{12084}
Let $(b_0^L,u_{h,0}^L, u^L_{3,0})\in H^{m+3}(\mathbb{R}^2)$ with a nonnegative integer $m$. Then for any $\eps,\delta>0$ satisfying \eqref{lim}, the system \eqref{11183} admits a unique solution $(a,w_h,w_3)\in C([0,T^*];H^{m+3}(\mathbb{R}^2))$ for some $T^*>0$, and satisfies
\begin{equation}\label{12092}
\sup\limits_{t\in[0,T^*]}\|(a,w_h,w_3)(t)\|_{H^{m+3}(\mathbb{R}^2)}\leq C\|(b_0^L,u_{h,0}^L,u^L_{3,0})\|_{H^{m+3}(\mathbb{R}^2)},
\end{equation}
where the time $T^*$ and the positive constant $C$ are independent of $\delta$.
\end{lemma}

The main results of this paper are stated in the following two theorems:

\begin{theorem}\label{th-ex}$\mathrm{(Local~well\mbox{-}posedness~and~uniform~estimates).}$
Let $U_{0}$ be given by \eqref{in}-\eqref{eq161} with given integer $m\geq3$. Then, there exists a positive time $T$ such that for any $\delta>0$, the Cauchy problem of \eqref{CEL} with initial value $U_{0}$ admits a unique solution $U$ on $[0,T]$ with the form
\begin{equation}\label{eq1101}
U(t,x)=(a,w)^{\top}(t,x_{h})+(\vartheta,v)^{\top}(t,x),
\end{equation}
where $(a,w)^{\top}=(a,w_h,w_3)^{\top}\in C([0,T];H^{m+3}(\mathbb{R}^{2}))$ and $(\vartheta,v)^{\top}\in C([0,T];H^{m}(\mathbb{R}^{3}))$ are solutions to \eqref{11183} and \eqref{11191}, respectively. Moreover, there exists a positive constant $C=C(m,U_0)$, independent of $\delta$, such that
\begin{equation*}
\sup\limits_{t\in[0,T]}\|(a,w)(t)\|_{H^{m+3}(\mathbb{R}^2)}+\sup\limits_{t\in[0,T]}\|(\vartheta,v)(t)\|_{H^{m}(\mathbb{R}^3)}\leq C.
\end{equation*}
\end{theorem}

\begin{theorem}\label{th-main} $\mathrm{(Convergence).}$
Let $m\geq3$ and
\begin{equation*}
U(t,x)=(a,w_h,w_3)^{\top}(t,x_{h})+(\vartheta,v)^{\top}(t,x)\in C([0,T];H^{m+3}(\mathbb{R}^{2}))+C([0,T];H^{m}(\mathbb{R}^{3}))
\end{equation*}
be the solution to \eqref{CEL}-\eqref{in} given by Theorem \ref{th-ex}. Then for any $q\in(2,\infty)$,
\begin{equation*}
\|(\vartheta,v)\|_{L^q(0,T;W^{m-3,\infty}(\mathbb{R}^3))}\leq C\delta^{\frac{1}{q}},
\end{equation*}
and $(a,w_h,w_3)^{\top}$ has a decomposition $(a,w_h,w_3)^{\top}=(a^S,w^S_h,w_3)^{\top}+(a^F,w^F_h,0)^{\top}$ such that
\begin{equation*}
\begin{split}
&\|a^S-b^L\|_{L^{\infty}(0,T;H^{m}(\mathbb{R}^2))}+\|(w^S_h-u^L_h,w_3-u^L_3)\|_{L^{\infty}(0,T;H^{m-1}(\mathbb{R}^2))}\\
&+\|(a^F,w^F_h)\|_{L^q(0,T;W^{m-1,\infty}(\mathbb{R}^2))}\leq C\delta^{\frac{1}{q}},
\end{split}
\end{equation*}
where $(b^L,u^L_h,u^L_3)\in C([0,T];H^{m+3}(\mathbb{R}^2))$ is the strong solution to the equations \eqref{Lim-sys}, and the positive constants $C$ are independent of $\delta$.

\end{theorem}

\begin{remark}
Let us provide some comments on the main theorems.
\begin{enumerate}[(1)]
    \item{
    In their seminal work, J.-Y. Chemin et al. \cite{J-Y} considered the incompressible rotating Navier--Stokes system with initial data being the sum of a $2$D part and a $3$D part. Based on the Strichartz estimate, they proved the convergence from the original system to the $2$D incompressible Navier--Stokes system when the Rossby number goes to zero. However, in the context of compressible inviscid fluids, there seems no previous study handling the initial data of the sum form.
    }
    \item{
    In the particular case $P(\rho)=\rho^2$, our calculations in Section \ref{sec:QG} provide a rigorous proof from the $2$D compressible rotating shallow water system to the $2$D QG system, which is of independent interest.
    }
    
    \item {By the continuation methods developed by \cite{B4265,b-Takada-2016-Jap}, we believe that the existence times $T^*$ and $T$ may be chosen arbitrarily large provided that $\delta$ is sufficiently small. However, since our aim is the singular limit of the solution around a two-dimensional flow, we shall not go any deeper in the direction of long-time existence.
    }
\end{enumerate}
\end{remark}

We give some notations which will be used in the sequel. For a vector field $v=(v_1,v_2,v_3)$, we use $v_h=(v_1,v_2)$ to denote the horizontal components. For a function $f$, we use $\widehat{f}$ to denote its Fourier transform, and use $\mathcal{F}^{-1}f$ to denote its inverse. We define $\sigma(D)f(x):=\mathcal{F}^{-1}(\sigma(\xi)\widehat{f}(\xi))$ for $\sigma=\sigma(\xi)$, and use $C$ to denote generic positive constants changing from line to line.

\section{Local well-posedness and dispersion of the 3D perturbed system}\label{sec2}
In this section, we prove the local well-posedness of strong solutions to the $3$D perturbed system \eqref{11191} which yields Theorem \ref{th-ex}. Then we show that the unique solution to \eqref{11191} reduces to zero as {\color{blue}$\delta\to0$}. By setting $V:=(\vartheta,v)^{\top}$, we may rewrite \eqref{11191} as
\begin{equation}\label{12031}
\partial_t V+\frac{\bar{\gamma}}{\delta}\mathcal{L} V+\mathcal{N}(V,\nabla V)=\mathcal{G}(V),~~~V|_{t=0}=V_0:=(b_0,u_0)^{\top},
\end{equation}
where
\begin{equation*}
\mathcal{G}(V)=-(w\cdot\nabla) V-
\begin{pmatrix}
\bar{\gamma} a\operatorname{div}v\\\bar{\gamma}a\nabla \vartheta
\end{pmatrix}
-
\begin{pmatrix}
v\cdot\nabla a\\(v\cdot\nabla) w
\end{pmatrix}
-
\begin{pmatrix}
\bar{\gamma}\vartheta \operatorname{div}w\\
\bar{\gamma}\vartheta \nabla a
\end{pmatrix}
,
\end{equation*}
and $\mathcal{L}$, $\mathcal{N}(\cdot,\cdot)$ are defined in \eqref{e9182}.

\begin{theorem}\label{Exis-Mid}
Let $m\geq 3$ be an integer and let $(b_0,u_0)\in H^{m}(\mathbb{R}^3)$, and $(a,w)\in C([0,T^*];H^{m+3}(\mathbb{R}^2))$ for some $0<T^*<\infty$. Then there exists a positive time $T\in(0,T^*]$ which is independent of $\delta$, such that the system \eqref{12031} admits a unique solution $V$ on $[0,T]$, belongs to $C([0,T];H^{m}(\mathbb{R}^3))$, and satisfies
\begin{equation}\label{12032}
\sup\limits_{t\in[0,T]}\|V(t)\|_{H^{m}(\mathbb{R}^3)}\leq C\big(\|(b_0,u_0)\|_{H^{m}(\mathbb{R}^3)},\|(a,w)\|_{L^{\infty}(0,T^*;W^{m+1,\infty}(\mathbb{R}^2))},T\big).
\end{equation}
Furthermore, for any $q\in(2,\infty)$, it holds
\begin{equation}\label{12061}
\|V\|_{L^q(0,T;W^{m-3,\infty}(\mathbb{R}^3))}\leq C\big(\|(b_0,u_0)\|_{H^{m}(\mathbb{R}^3)},\|(a,w)\|_{L^{\infty}(0,T^*;W^{m+1,\infty}(\mathbb{R}^2))},T\big)\delta^{\frac{1}{q}}.
\end{equation}
\end{theorem}

The proof of Theorem \ref{Exis-Mid} will be given in the following two subsections. Obviously, combining Theorem \ref{Exis-Mid}, Lemma \ref{12084} and the continuous embedding $H^{m+3}(\mathbb{R}^2)\hookrightarrow W^{m+1,\infty}(\mathbb{R}^2)$, we conclude Theorem \ref{th-ex} readily.

\subsection{Local well-posedness}\label{sec-well}
In this subsection, we prove the results of existence and uniform estimates stated in Theorem \ref{Exis-Mid}.
We begin from showing that \eqref{12032} holds for the solution to \eqref{12031}. Let $V\in C([0,T];H^{m}(\mathbb{R}^3))$ be the local strong solution to \eqref{12031}.
For any $\alpha \in ( \mathbb{N} \cup \{ 0\} )^3$ with $0\leq |\alpha|\leq m$, applying $D^{\alpha}$ to \eqref{12031},
taking the $L^2$ inner product of the resulting equations with $D^{\alpha}V$ and integrating by parts,
we get
\begin{align}
\frac{1}{2}\frac{d}{dt}\int_{\mathbb{R}^3}{|D^{\alpha}V|^2}dx=
&
-\int_{\mathbb{R}^3}{D^{\alpha}\mathcal{N}(V,\nabla V)\cdot D^{\alpha}V}dx
-
\int_{\mathbb{R}^3}{D^{\alpha}((w\cdot\nabla) V)\cdot D^{\alpha}V}dx
 \nonumber  \\
&-\bar{\gamma}\int_{\mathbb{R}^3}{
\Big(
D^{\alpha}(a\operatorname{div}v) D^{\alpha} \vartheta+D^{\alpha}(a\nabla \vartheta)\cdot D^{\alpha}v
\Big)
}dx  \nonumber \\
&
-\int_{\mathbb{R}^3}{
\Big(
D^{\alpha}(v\cdot\nabla a) D^{\alpha}\vartheta+D^{\alpha}((v\cdot\nabla) w)\cdot D^{\alpha}v
\Big)
} dx
 \nonumber   \\
&-\bar{\gamma}\int_{\mathbb{R}^3}{
\Big(
D^{\alpha}(\vartheta\cdot\operatorname{div}w) D^{\alpha}\vartheta+D^{\alpha}(\vartheta\cdot\nabla a)\cdot D^{\alpha}v
\Big)
}dx,  \label{12033}
\end{align}
where we have used that
\begin{align*}
    \int_{\mathbb{R}^3} D^{\alpha} \mathcal{L} V \cdot D^{\alpha}  V dx=0.
\end{align*}  
Notice that $\mathcal{N}(V,\nabla V)$ is the nonlinear part of the $3$D compressible Euler equations. Thus, by using the classical estimates for symmetric hyperbolic system as in \cite{M-book} we obtain
\begin{equation}\label{12036}
\int_{\mathbb{R}^3}{D^{\alpha}\mathcal{N}(V,\nabla V)\cdot D^{\alpha}V}dx\leq C
\|V\|_{H^{m}(\mathbb{R}^3)}^3   .
\end{equation}
Then, integrating by parts gives
\begin{equation}\label{12037}
\begin{split}
&\int_{\mathbb{R}^3}{D^{\alpha}(w\cdot\nabla V)\cdot D^{\alpha}V}dx\\
&=\sum\limits_{\substack{\alpha_1+\alpha_2=\alpha\\|\alpha_1|\geq1}}C(\alpha_1,\alpha_2)
\int_{\mathbb{R}^3}{(D^{\alpha_1}w\cdot D^{\alpha_2}\nabla) V\cdot D^{\alpha}V}dx-\frac{1}{2}
\int_{\mathbb{R}^3}{\operatorname{div}_h w_h |D^{\alpha}V|^2}dx\\
&\leq C\|w\|_{W^{m,\infty}(\mathbb{R}^2)}\|V\|_{H^m(\mathbb{R}^3)}^2.
\end{split}
\end{equation}
In a similar manner,
\begin{equation}\label{12038}
\begin{split}
&\bar{\gamma}
\int_{\mathbb{R}^3}{
\Big(
D^{\alpha}(a\operatorname{div}v) D^{\alpha} \vartheta+D^{\alpha}(a \nabla \vartheta)\cdot D^{\alpha}v
\Big)
}dx\\
&=\bar{\gamma}\sum\limits_{\substack{\alpha_1+\alpha_2=\alpha\\|\alpha_1|\geq1}}C(\alpha_1,\alpha_2)
\int_{\mathbb{R}^3}{
\Big(
D^{\alpha_1}a D^{\alpha_2}\operatorname{div}v D^{\alpha}\vartheta+D^{\alpha_1}a D^{\alpha_2}\nabla \vartheta\cdot D^{\alpha}v
\Big)
}dx\\
&~~~~-\bar{\gamma}\int_{\mathbb{R}^3}{(\nabla_h a\cdot D^{\alpha}v_h) D^{\alpha}\vartheta}dx\\
&\leq C\|a\|_{W^{m,\infty}(\mathbb{R}^2)}\|V\|_{H^m(\mathbb{R}^3)}^2.
\end{split}
\end{equation}
Finally, direct calculations show that the last two integrals in \eqref{12033} are bounded by
\begin{equation}\label{12039}
C\|(\nabla a,\nabla w)\|_{W^{m,\infty}(\mathbb{R}^2)}\|V\|^2_{H^m(\mathbb{R}^3)}.
\end{equation}
Combining the estimates at hand and applying Gronwall's inequality we obtain \eqref{12032} for some $T>0$ suitably small but is independent of $\delta$.

Now we prove the existence part of Theorem \ref{Exis-Mid}. Indeed, equation \eqref{12031} can be approximated by
\begin{equation}\label{12034}
\partial_t V^{\beta}+\frac{\bar{\gamma}}{\delta}\mathcal{L} J_{\beta}V^{\beta}+J_{\beta}\mathcal{N}(J_{\beta}V^{\beta},\nabla J_{\beta} V^{\beta})=J_{\beta}\mathcal{G}(J_{\beta}V^{\beta}),~~~V^{\beta}|_{t=0}=J_{\beta}V_0,
\end{equation}
where $J_{\beta}$ is a standard symmetric mollifier that tends to the identity operator as $\beta\to0$. Existence of solutions to the system \eqref{12034} on some time interval $[0,T_{\delta,\beta}]$ for fixed positive $\delta$ and $\beta$ follows from the existence theorem for ODEs on Banach spaces, and \eqref{12034} is equivalent to
\begin{equation}\label{12035}
\partial_t V^{\beta}+\frac{\bar{\gamma}}{\delta}\mathcal{L} V^{\beta}+J_{\beta}\mathcal{N}(V^{\beta},\nabla  V^{\beta})=J_{\beta}\mathcal{G}(V^{\beta}),~~~V^{\beta}|_{t=0}=J_{\beta}V_0,
\end{equation}
by observing that $V^{\beta}=J_{\beta}V^{\beta}$. Now, for $0\leq |\alpha|\leq m$, applying $D^{\alpha}$ to \eqref{12035}, taking the $L^2$ inner product of the resulting equations with $D^{\alpha}V^{\beta}$, using similar arguments of \eqref{12036}-\eqref{12039} and applying Gronwall's inequality, we finally obtain
\begin{equation}\label{120310}
\sup\limits_{t\in[0,T]}\|V^{\beta}(t)\|_{H^{m}(\mathbb{R}^3)}\leq C\big(\|(b_0,u_0)\|_{H^{m}(\mathbb{R}^3)},\|(a,w)\|_{L^{\infty}(0,T^*;W^{m+1,\infty}(\mathbb{R}^2))},T\big)
\end{equation}
for some time $T\in(0,T^*]$ which is independent of $\delta$ and $\beta$. Meanwhile, it follows from \eqref{12035} and \eqref{120310} that
\begin{equation}\label{12041}
\sup\limits_{t\in[0,T]}\|\partial_tV^{\beta}(t)\|_{L^{2}(\mathbb{R}^3)}\leq C(\|(b_0,u_0)\|_{H^{m}(\mathbb{R}^3)},\|(a,w)\|_{L^{\infty}(0,T^*;W^{m+1,\infty}(\mathbb{R}^2))},T,\delta).
\end{equation}
Thus, combining \eqref{120310} and \eqref{12041} we get the strong convergence of the sequence $\{V^{\beta}\}_{\beta}$ as $\beta\to0$ (up to a subsequence). Passing to the limit $\beta\to0$ yields a strong solution to the equation \eqref{12031}. Based on the energy estimates, uniqueness of solution is verified similarly and the details are omitted.

\subsection{Strichartz estimates}
In this section, we establish the Strichartz decay estimates for the solution to \eqref{12031}. To this end, let $V\in C([0,T];H^{m}(\mathbb{R}^3))$ be the unique solution to \eqref{12031} constructed in Section \ref{sec-well}.

Let $\varphi\in \mathcal{S}(\mathbb{R}^{3})$ be a radial function, supported in the annulus $\mathcal{C}:=\{\xi\in\mathbb{R}^{3}:\frac{3}{4}\leq |\xi|\leq \frac{8}{3}\}$ such that
\begin{equation*}
\sum\limits_{k\in\mathbb{Z}}\varphi_{k}(\xi)=1,~~\forall\,\xi\neq0,
\end{equation*}
where $\varphi_{k}(\xi)=\varphi(2^{-k}\xi)$. We define the operators
\begin{equation}\label{235}
\begin{split}
&\Delta_{k}f(x):=\varphi_{k}(D)f(x),~~k\in\mathbb{Z}.
\end{split}
\end{equation}
For $(r,\sigma)\in[1,\infty]^{2}$, $m\in\mathbb{R}$ and $q\in[1,\infty]$, we define the semi-norm of the standard homogeneous Besov spaces $\dot{B}^{m}_{r,\sigma}(\mathbb{R}^{3})$:
\begin{equation*}
\|f\|_{\dot{B}^{m}_{r,\sigma}(\mathbb{R}^{3})}:=\bigg(\sum\limits_{k\in\mathbb{Z}}\big(2^{mk}\|\Delta_{k}f\|_{L^{r}(\mathbb{R}^{3})}\big)^{\sigma}\bigg)^{\frac{1}{\sigma}},
\end{equation*}
and the space-time norm of Chemin--Lerner type:
\begin{equation}\label{eq:besovspacetime}
\|f\|_{\widetilde{L^{q}}(0,T;\dot{B}^{m}_{r,\sigma}(\mathbb{R}^{3}))}:=\bigg(\sum\limits_{k\in\mathbb{Z}}\big(2^{mk}\|\Delta_{k}f\|_{L^{q}(0,T;L^{r}(\mathbb{R}^{3}))}\big)^{\sigma}\bigg)^{\frac{1}{\sigma}}.
\end{equation}
We have
\begin{equation}\label{892}
\begin{split}
&\|f\|_{\widetilde{L^{q}}(0,T;\dot{B}^{m}_{r,\sigma}(\mathbb{R}^{3}))}\leq \|f\|_{L^{q}(0,T;\dot{B}^{m}_{r,\sigma}(\mathbb{R}^{3}))}~\mathrm{if}~q\leq\sigma,
\\
&\|f\|_{L^{q}(0,T;\dot{B}^{m}_{r,\sigma}(\mathbb{R}^{3}))}\leq \|f\|_{\widetilde{L^{q}}(0,T;\dot{B}^{m}_{r,\sigma}(\mathbb{R}^{3}))}~\mathrm{if}~\sigma\leq q.
\end{split}
\end{equation}
We refer to the monograph \cite{BB50} for more details about Besov spaces.

To derive the Strichartz-type decay estimate of solution to \eqref{12031}, we give the following lemma, which can be justified by  Lemma 3.2 in \cite{B-Fujii}.

\begin{lemma}\label{lem-Fujii}
Assume that $q,r$ satisfy
\begin{equation}\label{12051}
2\leq q,r\leq\infty,~~\frac{1}{q}+\frac{1}{r}\leq\frac{1}{2},~~(q,r)\neq(2,\infty).
\end{equation}
Then there exists a positive constant $C=C(q,r,\nu)$ such that
\begin{equation*}
\|\Delta_k V\|_{L^q(0,T;L^r(\mathbb{R}^3))}\leq
C 2^{3(\frac{1}{2}-\frac{1}{r}-\frac{1}{q})k}\delta^{\frac{1}{q}}\Big(\|\Delta_k V_0\|_{L^2(\mathbb{R}^3)}+\|\Delta_k(\mathcal{N}(V,\nabla V),\mathcal{G}(V))\|_{L^{1}(0,T;L^{2}(\mathbb{R}^3))}\Big)
\end{equation*}
holds for all $k\in\mathbb{Z}$ with $2^k\leq \delta/\varepsilon$, and
\begin{equation*}
\|\Delta_k V\|_{L^q(0,T;L^r(\mathbb{R}^3))}\leq
C 2^{3(\frac{1}{2}-\frac{1}{r})k}\delta^{\frac{1}{q}}\Big(\|\Delta_k V_0\|_{L^2(\mathbb{R}^3)}+\|\Delta_k(\mathcal{N}(V,\nabla V),\mathcal{G}(V))\|_{L^{1}(0,T;L^{2}(\mathbb{R}^3))}\Big)
\end{equation*}
holds for all $k\in\mathbb{Z}$ with $2^k> \delta/\varepsilon$.
\end{lemma}

Let $q,r$ satisfy the conditions in \eqref{12051}. Then from \eqref{892} and Lemma \ref{lem-Fujii}  we get that, for $\ell\in
 ( \mathbb{N} \cup \{ 0\} )  \cap[0,m-3]$,
\begin{equation*}
\begin{split}
&\|\nabla^{\ell} V\|_{L^q(0,T;L^{\infty}(\mathbb{R}^3))}
\leq
C
\|V\|_{L^q(0,T;\dot{B}^{\ell+\frac{3}{r}}_{r,1}(\mathbb{R}^3))}\leq C\|V\|_{\widetilde{L^q}(0,T;\dot{B}^{\ell+\frac{3}{r}}_{r,1}(\mathbb{R}^3))}\\
&~~~
\leq
C
\delta^{\frac{1}{q}}\sum\limits_{k\leq\log_{2}{\delta/\varepsilon}}2^{k(\ell +\frac{3}{2}-\frac{3}{q})}\Big(\|\Delta_k V_0\|_{L^2(\mathbb{R}^3)}+\|\Delta_k(\mathcal{N}(V,\nabla V),\mathcal{G}(V))\|_{L^{1}(0,T;L^{2}(\mathbb{R}^3))}\Big)\\
&~~~~~~
+C\delta^{\frac{1}{q}}\sum\limits_{k>\log_{2}{\delta/\varepsilon}}2^{k(\ell+\frac{3}{2})}\Big(\|\Delta_k V_0\|_{L^2(\mathbb{R}^3)}+\|\Delta_k(\mathcal{N}(V,\nabla V),\mathcal{G}(V))\|_{L^{1}(0,T;L^{2}(\mathbb{R}^3))}\Big)\\
&~~~
\leq
C
\delta^{\frac{1}{q}}\Big(\|V_0\|_{\dot{B}^{\ell+\frac{3}{2}-\frac{3}{q}}_{2,1}\cap \dot{B}^{\ell+\frac{3}{2}}_{2,1}(\mathbb{R}^3)}+\|(\mathcal{N}(V,\nabla V),\mathcal{G}(V))\|_{L^1(0,T;\dot{B}^{\ell+\frac{3}{2}-\frac{3}{q}}_{2,1}\cap\dot{B}^{\ell+\frac{3}{2}}_{2,1} (\mathbb{R}^3))}\Big)\\
&~~~
\leq C\delta^{\frac{1}{q}}\Big(\|V_0\|_{H^{\ell+2}(\mathbb{R}^3)}+\|(\mathcal{N}(V,\nabla V),\mathcal{G}(V))\|_{L^1(0,T;H^{\ell+2}(\mathbb{R}^3))} \Big)\\
&~~~
\leq C\delta^{\frac{1}{q}}
\Big(
\|V_0\|_{H^{\ell+2}(\mathbb{R}^3)}+\|V\|^2_{L^2(0,T;H^{\ell+3}(\mathbb{R}^3))}   \\
&~~~~~~ \qquad \qquad
+\|(a,w)\|_{L^\infty(0,T;W^{\ell+3,\infty}(\mathbb{R}^2))}\|V\|_{L^1(0,T;H^{\ell+3}(\mathbb{R}^3))}
\Big),
\end{split}
\end{equation*}
which implies \eqref{12061}. We have finished the proof of Theorem \ref{Exis-Mid}.

\section{QG limit of 2D compressible rotating Euler equations}\label{sec:QG}
In this section,
we investigate the asymptotic limit of the intermediate system \eqref{11183}.
The key ingredient is to study the QG limit of the $2$D compressible rotating Euler equations.
The discussions are restricted in the $2$D horizontal plane. To distinguish the notations, in this section, we use $y=(y_1,y_2)\in\mathbb{R}^2$ to denote the spatial variables in horizontal directions, and use $\eta=(\eta_1,\eta_2)\in\mathbb{R}^2$ to denote the corresponding frequency variables.
For $f=f(y)$, the notations $\mathcal{F}(f)(\eta)$ and $\widehat{f}(\eta)$ still denote its Fourier transform, and $\mathcal{F}^{-1}$ is the inverse of $\mathcal{F}$. For $g=g(\eta)$, we define the operator
\begin{equation*}
g(D_h)f(y):=\mathcal{F}^{-1}(g\mathcal{F}(f))(y).
\end{equation*}

Let $W:=(a,w_h)^{\top}$, then we may rewrite \eqref{11183}$_{1,2}$ as
\begin{equation}\label{12054}
\partial_t W+\frac{\bar{\gamma}}{\delta}\mathcal{A}W+\widetilde{\mathcal{N}}(W,\nabla_h W)=0,~~W|_{t=0}=W_0:=(b_0^L,u_{h,0}^L)^{\top},
\end{equation}
where
\begin{equation}\label{12055}
\mathcal{A}W=
\begin{pmatrix}
\operatorname{div}_h w_h\\
\nu w_h^{\bot}+\nabla_h a
\end{pmatrix}
,~~\mathrm{and}~~
\widetilde{\mathcal{N}}(W,\nabla_h W)
=
\begin{pmatrix}
w_h\cdot\nabla_h a+\bar{\gamma}a\operatorname{div}_h w_h\\
(w_h\cdot\nabla_h) w_h+\bar{\gamma}a\nabla_h a
\end{pmatrix}
.
\end{equation}
Taking the Fourier transform of the large operator $\mathcal{A}$ gives
\begin{equation*}
\widehat{\mathcal{A}}(\eta)=\left(\begin{matrix}
0&i\eta_{1}&i\eta_{2}\\
i\eta_{1}&0&-\nu\\
i\eta_{2}&\nu&0\\
\end{matrix}\right).
\end{equation*}
By solving the eigenvalue problems
\begin{equation*}
\widehat{\mathcal{A}}(\eta)\mathbf{d}^{}_{j}(\eta)=i p_{j}(\eta)\mathbf{d}_{j}(\eta),~~j\in\{0,+,-\},
\end{equation*}
we obtain
\begin{equation}\label{7252}
 p_0(\eta)=0,~~\mathrm{and}~~p_{\pm}(\eta)=\pm p(\eta),~~\mathrm{where}~~p(\eta)=\sqrt{\nu^2+|\eta|^2}.
\end{equation}
The exact expressions of the eigenvectors $\mathbf{d}_{j}(\eta)$ are not important for our purpose. However, we may choose in a way such that for any $\eta\in\mathbb{R}^2$, $\{\mathbf{d}_{j}(\eta)\}_{j=0,\pm}$ form an orthonormal basis in $\mathbb{C}^{3}$.
Now we define the projection operators
\begin{equation}\label{274}
\mathcal{P}_{(j)}f(y)=\mathcal{F}^{-1}\big((\widehat{f}\cdot \overline{\mathbf{d}_{j}})\mathbf{d}_{j}\big)(y),~~j\in\{0,+,-\}.
\end{equation}
Then, for any $f \in (L^2(\mathbb{R}^2))^3$, we see that $f=\sum\nolimits_{j=0,\pm}\mathcal{P}_{(j)}f$.
Applying $\mathcal{P}_{(j)}$ to \eqref{12054} and using Duhamel's principle yields, for $j\in\{0,+,-\}$,
\begin{equation}\label{12071}
\mathcal{P}_{(j)}W(t)=e^{-i\frac{\bar{\gamma}t}{\delta} p_{j}(D_h)}\mathcal{P}_{(j)}W_{0}-\int_{0}^{t}{e^{-i\bar{\gamma}\frac{t-s}{\delta} p_{j}(D_h)}\mathcal{P}_{(j)}\widetilde{\mathcal{N}}(W,\nabla_h W)(s)}ds.
\end{equation}
Using the above projections, we decompose $W=W^S+W^F$ with
\begin{equation}\label{12073}
W^S=
\begin{pmatrix}
a^S\\w_h^S
\end{pmatrix}
:=
\mathcal{P}_{(0)}W,~~\mathrm{and}~~
W^F=
\begin{pmatrix}
a^F\\w_h^F
\end{pmatrix}
:=
\mathcal{P}_{(-)}W+\mathcal{P}_{(+)}W.
\end{equation}
It is clear that $W^F$ and $W^S$ are respectively the so-called fast oscillating wave and the slow wave in the system \eqref{12054}. Direct calculations show
\begin{equation}\label{12074}
\nu w_h^{S,\bot}+\nabla_h a^S=0,~~\operatorname{div}_h w_h^S=0,~~\mathrm{and}~~\mathrm{curl}_h w_h^F-\nu a^F=0.
\end{equation}
Indeed, if $\mathbf{z}=(z_1,z_2,z_3)^{\top}$ such that $\widehat{\mathcal{A}}(\eta)\mathbf{z}=\lambda \mathbf{z}$ with $\lambda\neq0$, then we can check directly that
\begin{equation*}
i\eta_1 z_3-i\eta_2 z_2-\nu z_1=0.
\end{equation*}
Thus the property of $(a^F,w_h^F)$ in \eqref{12074} holds. The other one is verified similarly. 

The main result of this section is stated as follows, of which the proofs are presented in the following two subsections.

\begin{theorem}\label{th-2d}
Assume that $(b_0^L,u_{h,0}^L,u^L_{3,0})\in H^{m+3}(\mathbb{R}^2)$ with $m\geq 3$. Let $(W,w_3)=(a,w_h,w_3)\in C([0,T^*];H^{m+3}(\mathbb{R}^2))$ be the solution to \eqref{11183} ensured by Lemma \ref{12084}, and $a^S$, $w^S_h$, $W^F$ be defined in \eqref{12073}. Then for any $q\in(2,\infty)$,
\begin{equation}
\begin{split}\label{12091}
&\|a^S-b^L\|_{L^{\infty}(0,T^*;H^{m}(\mathbb{R}^2))}+\|(w^S_h-u^L_h,w_3-u^L_3)\|_{L^{\infty}(0,T^*;H^{m-1}(\mathbb{R}^2))}\leq C\delta^{\frac{1}{q}},\\
&
\|W^F\|_{L^q(0,T^*;W^{m-1,\infty}(\mathbb{R}^2))}\leq C\delta^{\frac{1}{q}},
\end{split}
\end{equation}
where $(b^L,u^L_h, u^L_3)$ is the solution to the target system \eqref{Lim-sys}, and the positive constants $C$ are independent of $\delta$.
\end{theorem}

\subsection{Strichartz estimates for the fast wave }
Let $\psi\in \mathcal{S}(\mathbb{R}^{2})$ be a radial function, supported in the annulus $\mathcal{C}:=\{\eta\in\mathbb{R}^{2}:\frac{3}{4}\leq |\eta|\leq \frac{8}{3}\}$ such that
\begin{equation*}
\sum\limits_{k\in\mathbb{Z}}\psi_{k}(\eta)=1,~~\forall\,\eta\neq0,
\end{equation*}
where $\psi_{k}(\eta)=\psi(2^{-k}\eta)$. We define the operators
\begin{equation}\label{235-1}
\begin{split}
&\Delta^h_{k}f(y):=\psi_{k}(D_h)f(y),~~k\in\mathbb{Z}.
\end{split}
\end{equation}
Then, it holds
\begin{equation*}
e^{-i\frac{\bar{\gamma}t}{\delta}p_{\pm}(D_h)}f=e^{\mp i\frac{\bar{\gamma}t}{\delta}p(D_h)}f=\sum\limits_{k\in\mathbb{Z}}e^{\mp i\frac{\bar{\gamma}t}{\delta}p(D_h)}\Delta_k^h f.
\end{equation*}

\begin{proposition}\label{pro-dis}
Let $p(\eta)$ be defined in \eqref{7252}. Then, for any $k\in\mathbb{Z}$ we have
\begin{equation*}
\|e^{\pm i\frac{\bar{\gamma}t}{\delta}p(D_h)}\Delta_k^h f\|_{L^{\infty}(\mathbb{R}^2)}\leq \frac{C2^{2k}}{1+\frac{|t|}{\mathcal{M}_k\delta}}\|\Delta_k^h f\|_{L^1(\mathbb{R}^2)},
\end{equation*}
where
\begin{equation}\label{12172}
\mathcal{M}_k=
\begin{cases}
1,~~\mathrm{if}~~k\leq\log_{2}{\nu};\\
2^{3k},~~\mathrm{if}~~k>\log_{2}{\nu},
\end{cases}
\end{equation}
and the positive constant $C$ is independent of $\delta$, $f$, $k$ and $t$.
\end{proposition}

\begin{proof}
We only provide the proof for the sign `$+$' since the other is obtained similarly. Let $\tilde{\psi}\in\mathcal{S}(\mathbb{R}^{2})$ be a radial function such that
\begin{equation*} 
0\leq\tilde{\psi}\leq1,~~\operatorname{supp}\,\tilde{\psi}\subset\left\{\eta\in\mathbb{R}^{2}:\frac{1}{2}\leq|\eta|\leq3\right\},~~\tilde{\psi}\equiv1~~\mathrm{on}~~\mathcal{C}=\left\{\eta\in\mathbb{R}^{2}:\frac{3}{4}\leq |\eta|\leq \frac{8}{3}\right\},
\end{equation*}
and for $k\in\mathbb{Z}$ we set $\tilde{\psi}_{k}(\eta):=\tilde{\psi}(2^{-k}\eta)$. Then we may rewrite
\begin{equation*}
e^{i\frac{\bar{\gamma}t}{\delta}p(D_h)}\Delta_k^h f=\frac{1}{(2\pi)^2}\int_{\mathbb{R}^2}{e^{i y\cdot\eta+i\frac{\bar{\gamma}t}{\delta}p(\eta)}\tilde{\psi}_k(\eta)\psi_k(\eta)\widehat{f}(\eta)}d\eta.
\end{equation*}
Thus, by using the convolution inequality and a stretching on the variable $\eta$, it suffices to show that
\begin{equation*}
\sup\limits_{y\in\mathbb{R}^2}\bigg|\int_{\mathbb{R}^2}{e^{i y\cdot\eta+i\frac{\bar{\gamma}t}{\delta}p(\eta)}\tilde{\psi}_k(\eta)}d\eta\bigg|=\sup\limits_{y\in\mathbb{R}^2}\bigg|2^{2k}\int_{\mathbb{R}^2}{e^{i y\cdot2^{k}\eta+i2^k\frac{\bar{\gamma}t}{\delta}q(\eta)}\tilde{\psi}(\eta)}d\eta\bigg|\leq \frac{C2^{2k}}{1+\frac{|t|}{\mathcal{M}_k\delta}},
\end{equation*}
where
\begin{equation*}
q(\eta)=\sqrt{\sigma_k^2+|\eta|^2},~~\mathrm{and }~~\sigma_k=\frac{\nu}{2^k}.
\end{equation*}
As the $L^{\infty}$ norm is invariant under dilation, it reduces to prove that
\begin{equation}\label{12171}
\sup\limits_{y\in\mathbb{R}^2}|I_k(y)|\leq \frac{C2^{2k}}{1+\frac{|t|}{\mathcal{M}_k\delta}},
\end{equation}
where
\begin{equation*}
I_k(y)=2^{2k}\int_{\mathbb{R}^2}{e^{i\theta_k\phi(\eta)}\tilde{\psi}(\eta)}d\eta,~~\mathrm{with}~~\theta_k=2^k\frac{\bar{\gamma} t}{\delta},~~\phi(\eta)=y\cdot\eta+q(\eta).
\end{equation*}

By using a partition of unity to cover the support of the original $\tilde{\psi}$, $\mathrm{i.e.}$, $\{1/2\leq |\eta|\leq 3\}$, we may assume that the support of $\tilde{\psi}$ is sufficiently small such that for any $\tilde{\eta},\eta\in\mathrm{supp}\,\tilde{\psi}$, the line segment connecting $\tilde{\eta}$ and $\eta$ lies wholly in $\{1/4<|\eta|<4\}$. It holds that
\begin{equation*}
\begin{split}
|I_k|^2=I_k\bar{I}_k&=2^{4k}\iint_{\mathbb{R}^2 \times \mathbb{R}^2}{e^{i\theta_k(\phi(\tilde{\eta})-\phi(\eta))}\tilde{\psi}(\tilde{\eta})\bar{\tilde{\psi}}(\eta)}d\tilde{\eta} d\eta\\
&=2^{4k}\iint_{\mathbb{R}^2 \times \mathbb{R}^2}{e^{i\theta_k(\phi(\eta+\tau)-\phi(\eta))}\tilde{\psi}(\eta+\tau)\bar{\tilde{\psi}}(\eta)}d\eta d\tau.
\end{split}
\end{equation*}
Let us define the operator $L$ by
\begin{equation*}
Lf(\eta)=\frac{1}{i\theta_k}\big(b\cdot\nabla_{\eta} f\big)(\eta),~~\mathrm{where}~~b=\frac{\nabla_{\eta}\phi(\eta+\tau)-\nabla_{\eta}\phi(\eta)}{|\nabla_{\eta}\phi(\eta+\tau)-\nabla_{\eta}\phi(\eta)|^{2}}.
\end{equation*}
Then one sees that $L^{\top}f=-\frac{1}{i\theta_k}\nabla_{\eta}\cdot(bf)$ and
\begin{equation}\label{742}
\begin{split}
|I_k|^2&=2^{4k}\iint_{\mathbb{R}^2 \times \mathbb{R}^2}{L^N\Big(e^{i\theta_k(\phi(\eta+\tau)-\phi(\eta))}\Big)\tilde{\psi}(\eta+\tau)\bar{\tilde{\psi}}(\eta)}d\eta d\tau\\
&=2^{4k}\iint_{\mathbb{R}^2 \times \mathbb{R}^2}{e^{i\theta_k(\phi(\eta+\tau)-\phi(\eta))} (L^{\top})^N\Big(\tilde{\psi}(\eta+\tau)\bar{\tilde{\psi}}(\eta)\Big)}d\eta d\tau
\end{split}
\end{equation}
for any $N\in\mathbb{N}$. Taking $N=3$ in \eqref{742} gives
\begin{equation}\label{12056}
|I_k|^2\leq C2^{4k}\iint_{\mathbb{R}^2 \times \mathbb{R}^2}{\frac{1}{|\theta_k|^3}\sum\limits_{j_1+j_2+j_3=0}^3|\nabla_{\eta}^{j_1}b|\cdot|\nabla_{\eta}^{j_2}b|\cdot|\nabla_{\eta}^{j_3}b|}d\eta d\tau.
\end{equation}
Direct calculations show that for any $\eta\in\{1/4<|\eta|<4\}$ it holds
\begin{equation*}
\begin{split}
&|\mathrm{det}\,\mathcal{D}^{2}q(\eta)|=\bigg|\frac{\sigma_k^2}{\sqrt{\sigma_k^2+|\eta|^2}}\bigg|\geq \mathcal{M}_{1,k}:=\begin{cases}
\sigma_k,~~\mathrm{if}~~\sigma_k\geq1;\\
\sigma_k^2,~~\mathrm{if}~~\sigma_k<1,
\end{cases}\\
&\mathrm{and}~|\nabla_{\eta}^{2} q|,~|\nabla_{\eta}^{3}q|,~|\nabla_{\eta}^{4}q|,~ |\nabla_{\eta}^{5}q|\leq C,
\end{split}
\end{equation*}
where $\mathcal{D}^{2}q$ is the Hessian matrix of $q$. By recalling the definition of $\phi$, it follows from the mean value theorem that
\begin{equation*}
|b|\leq\frac{C}{|\nabla_{\eta} q(\eta+\tau)-\nabla_{\eta} q(\eta)|}\leq\frac{C}{\inf_{1/4<|\eta|<4}|\mathrm{det}\,\mathcal{D}^{2}q(\eta)||\tau|}\leq \frac{C}{\mathcal{M}_{1,k}|\tau|},
\end{equation*}
In the same spirit, it is verified that
\begin{equation*}
\begin{split}
|\nabla_{\eta} b|&\leq \frac{C|\nabla_{\eta}^{2}q(\eta+\tau)-\nabla_{\eta}^{2}q(\eta)|}{|\nabla_{\eta} q(\eta+\tau)-\nabla_{\eta} q(\eta)|^{2}}\leq \frac{C\sup_{1/4<|\eta|<4}|\nabla_{\eta}^{3}q(\eta)|}{\inf_{1/4<|\eta|<4}|\mathrm{det}\,\mathcal{D}^{2}q(\eta)|^{2}|\tau|}\leq\frac{C}{\mathcal{M}_{1,k}^2|\tau|}
\end{split}
\end{equation*}
\begin{equation*}
\begin{split}
|\nabla_{\eta}^{2} b|&\leq\frac{C}{|\tau|}\Big(\frac{1}{\mathcal{M}^2_{1,k}}+\frac{1}{\mathcal{M}_{1,k}^3}\Big),~~\mathrm{and}~~|\nabla_{\eta}^{3} b|\leq \frac{C}{|\tau|}\Big(\frac{1}{\mathcal{M}^2_{1,k} }+\frac{1}{\mathcal{M}_{1,k}^3}+\frac{1}{\mathcal{M}_{1,k}^4}\Big).
\end{split}
\end{equation*}

With the estimates at hand, \eqref{12056} is further controlled by
\begin{equation*}
|I_k|^2\leq C2^{4k}\int\limits_{|\tau|\leq 4}{\frac{1}{|\theta_k\tau|^3}\Big(\frac{1}{\mathcal{M}_{1,k}^3}+\frac{1}{\mathcal{M}_{1,k}^6}\Big)}d\tau\leq
\begin{cases}
C2^{4k}\displaystyle\int\limits_{|\tau|\leq 4}{\frac{1}{|\theta_k\sigma_k\tau|^3}}d\tau,~~\mathrm{if}~~\sigma_k\geq 1;\\[12pt]
C2^{4k}\displaystyle\int\limits_{|\tau|\leq 4}{\frac{1}{|\theta_k\sigma_k^4\tau|^3}}d\tau,~~\mathrm{if}~~\sigma_k<1.
\end{cases}
\end{equation*}
Noting also that $|I_k|^2\leq C2^{4k}\int\nolimits_{|\tau|\leq 4}{1}d\tau$ by taking $N=0$ in \eqref{742}, thus
\begin{equation*}
|I_k|^2\leq \begin{cases}
C2^{4k}\displaystyle\int{\frac{1}{1+|\theta_k\sigma_k\tau|^3}}d\tau\leq \frac{C2^{4k}}{|\theta_k\sigma_k|^2}\leq C2^{4k}\Big(\frac{\delta}{|t|}\Big)^2,~~\mathrm{if}~~\sigma_k\geq1;\\[12pt]
C2^{4k}\displaystyle\int{\frac{1}{1+|\theta_k\sigma_k^4\tau|^3}}d\tau\leq \frac{C2^{4k}}{|\theta_k\sigma_k^4|^2}\leq C2^{4k}\Big(\frac{2^{3k}\delta}{|t|}\Big)^2,~~\mathrm{if}~~\sigma_k<1.
\end{cases}
\end{equation*}
\eqref{12171} follows from the above inequality readily, which completes the proof of Proposition \ref{pro-dis}.

\end{proof}

Using the dispersion property of the propagator, we may construct the following Strichartz-type decay estimate:

\begin{proposition}\label{pro-Str}
Let $2\leq r,\,\tilde{r}\leq \infty$, $2\leq q \leq \infty$ be subject to
\begin{equation*}
\frac{1}{q}+\frac{1}{r}\leq\frac{1}{2},~~\frac{1}{\tilde{q}}+\frac{1}{\tilde{r}}=\frac{1}{2},
~~(q,r)\neq(2,\infty),
~~(\tilde{q},\tilde{r})\neq(2,\infty).
\end{equation*}
Then, for any $\eps>0$, $\delta>0$, $f\in L^{2}(\mathbb{R}^{2})$, $F\in L^{\tilde{q}'}(\mathbb{R};L^{\tilde{r}'}(\mathbb{R}^{2}))$ and $k\in\mathbb{Z}$, it holds that
\begin{equation}\label{1910}
\|e^{\pm i\frac{\bar{\gamma}t}{\delta}p(D_h)}\Delta_k^h f\|_{L^q(\mathbb{R};L^{r}(\mathbb{R}^2))}\leq C2^{2k(\frac{1}{2}-\frac{1}{r})}(\mathcal{M}_k\delta)^{\frac{1}{q}}\|\Delta_k^h f\|_{L^{2}(\mathbb{R}^{2})},
\end{equation}
\begin{equation}\label{1911}
\bigg\|\int^{t}_{-\infty}{e^{\pm i\frac{\bar{\gamma}(t-s)}{\delta}p(D_h)}\Delta_k^h F(s)}ds\bigg\|_{L^q(\mathbb{R};L^{r}(\mathbb{R}^2))}\leq C2^{2k(\frac{1}{2}-\frac{1}{r}+\frac{1}{\tilde{q}})}(\mathcal{M}_k\delta)^{\frac{1}{q}+\frac{1}{\tilde{q}}}\|\Delta_k^h F\|_{L^{\tilde{q}'}(\mathbb{R};L^{\tilde{r}'}(\mathbb{R}^2))},
\end{equation}
where $1/\tilde{r}+1/\tilde{r}'=1$, $1/\tilde{q}+1/\tilde{q}'=1$, and the positive constants $C$ are independent of $\delta$, $f$, $F$ and $k$.

\end{proposition}

Indeed, Proposition \ref{pro-Str} can be justified by combining Proposition \ref{pro-dis} and the standard TT* argument with the help of the Hardy-Littlewood-Sobolev inequality, the interpolation theorem and the duality argument. We also refer the readers to Proposition 3.2 in \cite{Mu} for a similar proof. Now, using the continuous embedding $\dot{B}^{0}_{\infty,1}\hookrightarrow L^{\infty}$ and Proposition \ref{pro-Str} with $r=\infty$, $\tilde{q}=\infty$, $\tilde{r}=2$, we get, for any $q\in(2,\infty)$ and $\ell\in \mathbb{N} \cup \{ 0\} $,
\begin{equation}\label{12065}
\begin{split}
&\|\nabla_h^{\ell} e^{\pm i\frac{\bar{\gamma}t}{\delta} p(D_h)}f\|_{L^q(\mathbb{R};L^{\infty}(\mathbb{R}^2))}\\
&
\quad
\leq C\|e^{\pm i\frac{\bar{\gamma}t}{\delta} p(D_h)} f\|_{L^{q}(\mathbb{R};\dot{B}^{\ell}_{\infty,1}(\mathbb{R}^2))}\\
&
\quad
\leq C\sum\limits_{k\in\mathbb{Z}}2^{k\ell}\|e^{\pm i\frac{\bar{\gamma}t}{\delta} p(D_h)} \Delta^h_k f\|_{L^{q}(\mathbb{R};L^{\infty}(\mathbb{R}^2))}\\
&
\quad
\leq C\delta^{\frac{1}{q}}\bigg(\sum\limits_{k\leq \log_{2}{\nu}}2^{k(\ell+1)}\|\Delta_k^h f\|_{L^{2}(\mathbb{R}^2)}+\sum\limits_{k> \log_{2}{\nu}}2^{k(\ell+1+\frac{3}{q})}\|\Delta_k^h f\|_{L^{2}(\mathbb{R}^2)}\bigg)\\
&
\quad
\leq C\delta^{\frac{1}{q}}\|f\|_{\dot{B}^{\ell+1}_{2,1}\cap\dot{B}^{\ell+1+\frac{3}{q}}_{2,1} (\mathbb{R}^2)}\leq C\delta^{\frac{1}{q}}\|f\|_{H^{\ell+3}(\mathbb{R}^2)},
\end{split}
\end{equation}
and
\begin{equation}\label{12066}
\begin{split}
&\bigg\|\nabla^{\ell}_h\int^{t}_{-\infty}{e^{\pm i\frac{\bar{\gamma}(t-s)}{\delta} p(D_h)} F(s)}ds\bigg\|_{L^q(\mathbb{R};L^{\infty}(\mathbb{R}^2))}\\
&
\quad
\leq C\sum\limits_{k\in\mathbb{Z}}2^{k\ell}\bigg\|\int^{t}_{-\infty}{e^{\pm i\frac{\bar{\gamma}(t-s)}{\delta} p(D_h)} \Delta^h_k F(s)}ds\bigg\|_{L^q(\mathbb{R};L^{\infty}(\mathbb{R}^2))}\\
&
\quad
\leq C\delta^{\frac{1}{q}}\bigg(\sum\limits_{k\leq \log_{2}{\nu}}2^{k(\ell+1)}\|\Delta_k^h F\|_{L^1(\mathbb{R};L^{2}(\mathbb{R}^2))}+\sum\limits_{k> \log_{2}{\nu}}2^{k(\ell+1+\frac{3}{q})}\|\Delta_k^h F\|_{L^1(\mathbb{R};L^{2}(\mathbb{R}^2))}\bigg)\\
&
\quad
\leq C\delta^{\frac{1}{q}}\|F\|_{L^1(\mathbb{R};\dot{B}^{\ell+1}_{2,1}\cap\dot{B}^{\ell+1+\frac{3}{q}}_{2,1}(\mathbb{R}^2))}\leq C\delta^{\frac{1}{q}}\|F\|_{L^1(\mathbb{R};H^{\ell+3}(\mathbb{R}^2))}.
\end{split}
\end{equation}
Thus, combining \eqref{7252}, \eqref{12071}, \eqref{12073}, \eqref{12065} and \eqref{12066} we get, for $q\in(2,\infty)$ and $\ell \in   ( \mathbb{N} \cup \{ 0\} )    \cap[0,m-1]$,
\begin{equation}\label{12083}
\begin{split}
&\|\nabla_h^{\ell} W^F\|_{L^q(0,T^*;L^{\infty}(\mathbb{R}^2))}\leq \sum\limits_{j=\pm}\|\nabla_h^{\ell} \mathcal{P}_{(j)}W\|_{L^q(0,T^*;L^{\infty}(\mathbb{R}^2))}\\
&~~\leq C\sum\limits_{j=\pm}\bigg(\|\nabla_h^{\ell} e^{i\frac{\bar{\gamma}t}{\delta}p_j(D_h)}\mathcal{P}_{(j)}W_{0}\|_{L^q(0,T^*;L^{\infty}(\mathbb{R}^2))}\\
&~~~~~~~~+\bigg\|\nabla_h^{\ell} \int_{0}^{t}{e^{i\bar{\gamma}\frac{t-s}{\delta}p_j(D_h)}\mathcal{P}_{(j)}\widetilde{\mathcal{N}}(W,\nabla_h W)(s)}ds\bigg\|_{L^q(0,T^*;L^{\infty}(\mathbb{R}^2))}\bigg)\\
&~~\leq C\delta^{\frac{1}{q}}\big(\|W_0\|_{H^{\ell+3}(\mathbb{R}^2)}+\|\widetilde{\mathcal{N}}(W,\nabla_h W)\|_{L^1(0,T^*;H^{\ell+3}(\mathbb{R}^2))}\big)\\
&~~\leq C\delta^{\frac{1}{q}}\big(\|W_0\|_{H^{m+3}(\mathbb{R}^2)}+\|W\|^2_{L^2(0,T^*;H^{m+3}(\mathbb{R}^2))}\big).
\end{split}
\end{equation}
Thus, \eqref{12091}$_2$ is proved.

\subsection{Convergence of the slow part}
Now we are ready to prove the convergence of the slow part $W^S$ and finish the proof of Theorem \ref{th-2d}. Observe that $m\geq3$ here. By taking $\mathrm{curl}_h$\eqref{11183}$_2-\nu\times$\eqref{11183}$_1-$\eqref{Lim-sys}$_1$ and noticing that $\mathrm{curl}_h w_h-\nu a=\mathrm{curl}_h w_h^S-\nu a^S$, we get
\begin{equation}\label{12075}
\begin{split}
&\partial_t\varpi+w_h\cdot\nabla_h \varpi=\mathscr{R},
\end{split}
\end{equation}
where
\begin{equation*}
\begin{split}
&\varpi=\mathrm{curl}_h w_h^S-\nu a^S-(\mathrm{curl}_h u^L_{h}-\nu b^L),\\
&\mathscr{R}=-w^F_h\cdot\nabla_h(\mathrm{curl}_h u^L_h-\nu b^L)-\mathrm{curl}_h w_h \operatorname{div}_h w^F_h+\nu\bar{\gamma}a\operatorname{div}_h w_h^F\\
&~~~~~~-(w_h^S-u^L_h)\cdot\nabla_h(\mathrm{curl}_h u^L_h-\nu b^L).
\end{split}
\end{equation*}
By the standard energy estimate, we may infer for $t\in(0,T^*]$ that
\begin{equation}\label{12077}
\begin{split}
\frac{1}{2}\frac{d}{dt}\|\varpi(t)\|_{H^{m-2}(\mathbb{R}^2)}^2
&
\leq C\|w_h\|_{H^{m}(\mathbb{R}^2)}\|\varpi\|_{H^{m-2}(\mathbb{R}^2)}^2+C\|\mathscr{R}\|_{H^{m-2}(\mathbb{R}^2)}\|\varpi\|_{H^{m-2}(\mathbb{R}^2)}\\
&
\leq
C\|\varpi\|_{H^{m-2}(\mathbb{R}^2)}^2
+
C\Big(\|W^F\|_{W^{m-1,\infty}(\mathbb{R}^2)}+\|w^S_h-u^L_h\|_{H^{m-1}(\mathbb{R}^2)}\Big)  \\
& \qquad \qquad \qquad \qquad \qquad \qquad
\times
\|\varpi\|_{H^{m-2}(\mathbb{R}^2)},
\end{split}
\end{equation}
where we have used Theorem \ref{th-ex}. Thanks to
\begin{equation*}
\mathrm{curl}_h w^S_h=\frac{1}{\nu} \Delta_h a^S,~~\mathrm{curl}_h u^L_h=\frac{1}{\nu}\Delta_h b^L,~~\mathrm{and}~~\operatorname{div}_h w^S_h=\operatorname{div}_h u^L_h=0,
\end{equation*}
we have
\begin{equation}\label{12201}
\frac{1}{\nu}\Delta_h (a^S-b^L)-\nu(a^S-b^L)=\varpi.
\end{equation}
Thus
\begin{equation}\label{12076}
\begin{split}
\|w^S_h-u^L_h\|_{H^{m-1}(\mathbb{R}^2)}
&
\leq
C \|\nabla_h^{\bot}\Delta_h^{-1}\mathrm{curl}_h(w^S_h-u^L_h)\|_{H^{m-1}(\mathbb{R}^2)}\\
&
\leq
C \|\nabla_h^{\bot}(a^S-b^L)\|_{H^{m-1}(\mathbb{R}^2)}\leq C \|\varpi\|_{H^{m-2}(\mathbb{R}^2)}.
\end{split}
\end{equation}
Inserting \eqref{12076} into \eqref{12077}, dividing the resulting inequality by $\|\varpi\|_{H^{m-2}(\mathbb{R}^2)}$, applying Gronwall's inequality and using \eqref{12083}, we obtain for any $q\in(2,\infty)$ that
\begin{equation}\label{12101}
\begin{split}
\sup\limits_{t\in[0,T^*]}\|\varpi(t)\|_{H^{m-2}(\mathbb{R}^2)}&\leq C\|\varpi|_{t=0}\|_{H^{m-2}(\mathbb{R}^2)}+C\int^{T*}_0{\|W^F(s)\|_{W^{m-1,\infty}(\mathbb{R}^2)}}ds\leq C\delta^{\frac{1}{q}},
\end{split}
\end{equation}
where we have invoked the initial data in \eqref{Lim-sys}-\eqref{11183} and the property \eqref{12074}.
From the above inequality and \eqref{12201}-\eqref{12076}, we have for any $t\in[0,T^*]$ and $q\in(2,\infty)$ that
\begin{equation*}
\begin{split}
&\|(a^S-b^L)(t)\|_{H^{m}(\mathbb{R}^2)}+\|(w^S_h-u^L_h)(t)\|_{H^{m-1}(\mathbb{R}^2)}\leq C\|\varpi(t)\|_{H^{m-2}(\mathbb{R}^2)}\leq C\delta^{\frac{1}{q}}.
\end{split}
\end{equation*}
Finally, subtracting \eqref{Lim-sys}$_2$ from \eqref{11183}$_{3}$ we obtain
\begin{equation*}
\partial_t(w_3-u^L_3)+w_h\cdot\nabla_h (w_3-u^L_3)=-(w^S_h-u^L_h)\cdot\nabla_h u^L_3-w^F_h\cdot\nabla_h u^L_3.
\end{equation*}
Hence, similar to the derivation of \eqref{12101}, we conclude with
\begin{equation*}
\begin{split}
&\sup\limits_{t\in[0,T^*]}\|(w_3-u_3^L)(t)\|_{H^{m-1}(\mathbb{R}^2)}\\
&\quad
\leq \|(w_3-u_3^L)(0)\|_{H^{m-1}(\mathbb{R}^2)}+
C\int^{T^*}_{0}
{
\Big( \|(w^S_h-u^L_h)(s)\|_{H^{m-1}(\mathbb{R}^2)}+\|w_h^F(s)\|_{W^{m-1,\infty}(\mathbb{R}^2)}
\Big)
}ds\\
&
\quad
\leq C\delta^{\frac{1}{q}},~~\text{for all} \,q\in(2,\infty).
\end{split}
\end{equation*}
This completes the proof of Theorem \ref{th-2d}.

\vspace{5mm}
\centerline{\bf Acknowledgements}
\vspace{2mm}
Y. Li was supported by Natural Science Foundation of Anhui Province under grant number 2408085MA018, Natural Science Research Project in Universities of Anhui Province under grant number 2024AH050055.
P. Mu was supported by National Natural Science Foundation of China (Grant No. 12301279 and No. 12371229), Anhui Provincial Natural Science Foundation (Grant No. 2308085QA03), and Excellent University Research and Innovation Team in Anhui Province (Grant No. 2024AH010002).
\vspace{5mm}

\centerline{\bf Conflict of interest}
\vspace{2mm}
On behalf of all authors, the corresponding author states that there is no conflict of interest.

\vspace{5mm}

\centerline{\bf Data Availability Statement}
\vspace{2mm}
Data sharing not applicable to this article as no datasets were generated or
analyzed during the current study.

\end{document}